\documentclass[12pt]{amsart}
\usepackage{amssymb, amstext, amscd, amsmath}

\makeatletter
\def\@cite#1#2{{\m@th\upshape\bfseries%
[{#1\if@tempswa{\m@th\upshape\mdseries, #2}\fi}]}}
\makeatother

\let\secsymb=\S

\theoremstyle{plain}
\newtheorem{thm}{Theorem}[section]
\newtheorem{cor}[thm]{Corollary}

\newtheorem{lem}[thm]{Lemma}

\theoremstyle{definition}

\newtheorem{defn}[thm]{Definition}

\newtheorem{eg}[thm]{Example}

\newcommand{\bB}{{\mathbb{B}}}
\newcommand{\bC}{{\mathbb{C}}}
\newcommand{\bD}{{\mathbb{D}}}

\newcommand{\bT}{{\mathbb{T}}}
\newcommand{\bZ}{{\mathbb{Z}}}

  \newcommand{\A}{{\mathcal{A}}}
  \newcommand{\B}{{\mathcal{B}}}
  \newcommand{\C}{{\mathcal{C}}}
  
  \newcommand{\E}{{\mathcal{E}}}

\renewcommand{\H}{{\mathcal{H}}}
  
  \newcommand{\J}{{\mathcal{J}}}

  \newcommand{\M}{{\mathcal{M}}}
  \newcommand{\N}{{\mathcal{N}}}
\renewcommand{\O}{{\mathcal{O}}}

\renewcommand{\S}{{\mathcal{S}}}
  \newcommand{\T}{{\mathcal{T}}}
  \newcommand{\U}{{\mathcal{U}}}
  
  \newcommand{\W}{{\mathcal{W}}}

  \newcommand{\Z}{{\mathcal{Z}}}

\newcommand{\ep}{\varepsilon}
\renewcommand{\phi}{\varphi}
\newcommand{\upchi}{{\raise.35ex\hbox{\ensuremath{\chi}}}}

\newcommand{\fA}{{\mathfrak{A}}}
\newcommand{\fB}{{\mathfrak{B}}}

\newcommand{\fI}{{\mathfrak{I}}}
\newcommand{\fJ}{{\mathfrak{J}}}

\newcommand{\fL}{{\mathfrak{L}}}
\newcommand{\fM}{{\mathfrak{M}}}
\newcommand{\fR}{{\mathfrak{R}}}
\newcommand{\fT}{{\mathfrak{T}}}

\newcommand{\rC}{{\mathrm{C}}}

\newcommand{\qand}{\quad\text{and}\quad}

\newcommand{\qfor}{\quad\text{for}\quad}
\newcommand{\qforal}{\quad\text{for all}\quad}

\newcommand{\AND}{\text{ and }}

\newcommand{\Aut}{\operatorname{Aut}}
\newcommand{\diag}{\operatorname{diag}}

\newcommand{\Fix}{\operatorname{Fix}}
\newcommand{\id}{{\operatorname{id}}}

\newcommand{\ran}{\operatorname{Ran}}
\newcommand{\spn}{\operatorname{span}}

\newcommand{\rep}{\operatorname{rep}}
\newcommand{\srep}{\operatorname{s-rep}}

\newcommand{\ca}{\mathrm{C}^*}

\newcommand{\Fn}{\mathbb{F}_n^+}
\newcommand{\Fock}{\ell^2(\Fn)}
\newcommand{\lip}{\langle}
\newcommand{\rip}{\rangle}
\newcommand{\ip}[1]{\langle #1 \rangle}
\newcommand{\mt}{\varnothing}
\newcommand{\ol}{\overline}

\newcommand{\wot}{\textsc{wot}}
\newcommand{\ltwo}{\ell^2}

\begin{document}

\title[Semicrossed products]{Biholomorphisms of the unit ball
of $\bC^n$ and semicrossed products}

\author[K.R.Davidson]{Kenneth R. Davidson}
\address{Pure Math.\ Dept.\\U. Waterloo\\Waterloo, ON\;
N2L--3G1\\CANADA}
\email{krdavids@uwaterloo.ca}

\author[E.G.Katsoulis]{Elias~G.~Katsoulis}
\address{Dept. Math.\\ East Carolina University\\
Greenville, NC 27858\\USA}
\email{KatsoulisE@mail.ecu.edu}

\begin{abstract}
Assume that $\phi_1$ and $\phi_2$ are automorphisms of the non-commutative disc algebra $\fA_n$, $n \geq 2$.
We show that the semicrossed products $\fA_n \times_{\phi_1} \bZ^+$ and $\fA_n \times_{\phi_2} \bZ^+$ are isomorphic as algebras
if and only if $\phi_1$ and $\phi_2$ are conjugate
via an automorphism of $\fA_n$. A similar result holds for semicrossed
products of the d-shift algebra $\A_d$, $d \geq 2$.
\end{abstract}

\thanks{{\it Key words and phrases:}  }
\thanks{First author partially supported by an NSERC grant.}
\thanks{Second author was partially supported by a grant from ECU}

\date{}
\maketitle

\section{Introduction}\label{S:intro}

In this paper, we study crossed products of the non-commutative disk algebra $\fA_n$
and the commutative analogue, the $d$-shift algebra $\A_d$ of multipliers on the
Arveson--Drury space.  The isometric automorphisms of these algebras come from
the natural action of the group $\Aut(\bB_n)$ of conformal automorphisms of the
unit ball $\bB_n$ of $\bC^n$ on the character space $\ol{\bB_n}$.
We will show that the semi-crossed product that is formed determines
the automorphism up to analytic conjugacy.

There is an extensive body of work studying dynamical systems via associated
operator algebras going back to work of von Neumann.  The use of nonself-adjoint
operator algebras in this area begins with seminal work of Arveson \cite{Arv} and
Arveson--Josephson \cite{ArvJ}.  This was put into the abstract setting of
semi-crossed products by Peters \cite{Pet}.
See \cite{DKsurvey} for an overview of some of the recent work in this area.

Here we will be concerned with analytic actions on operator algebras which
encode that analytic structure in some way.
In \cite{HPW} , Hoover, Peters and Wogen initiated a study of semi-crossed products of
the disk algebra $A(\bD)$ by the automorphisms induced by a Mobius map.
It was subsequently studied by Buske and Peters \cite{BP}.

This problem was completed and extended in \cite{DKconj}.
If $K$ is the closure of a finitely connected bounded domain $\Omega$ in $\bC$
with nice boundary, then $A(K)$ denotes the space of continuous functions
on $K$ which are analytic on $\Omega$.
If $\gamma\in A(K)$ maps $K$ into itself, then it induces an endomorphism
of $A(K)$ sending $f\to f\circ\gamma$.
We showed that the semicrossed products $A(K_i)\times_{\gamma_i}\bZ^+$ are
isomorphic if and only if $\gamma_i$ are analytically conjugate,
except for one exceptional case.

Here we deal with a higher dimensional version.  The algebras $\fA_n$ (resp. $\A_n$)
are the universal norm-closed operator algebras for a (commuting) $n$-tuple of
operators $T_1,\dots,T_n$ such that $\| [ T_1\ \dots \ T_n] \| \le 1$.
This is in the sense that given any such $n$-tuple $T$, there is a unique
completely contractive homomorphism of $\fA_n$ (or $\A_n$) onto the
algebra $\A(T_1,\dots,T_n)$ taking generators to generators.
These algebras are more tractable, in many ways, than the algebras of analytic functions.

$\A_n$ is the complete quotient of $\fA_n$ by its commutator ideal.
So the two algebras are intimately linked.
The character space in both cases is homeomorphic to $\ol{\bB}_n$
and the Gelfand map carries them into the space $A(\bB_n)$ of
analytic functions on $\bB_n$ which extend to be continuous on the closure.
The group of isometric isomorphisms is just $\Aut(\bB_n)$, the group of
conformal automorphisms of $\bB_n$, and this action commutes with the Gelfand map.

In \cite{DKconj}, a lot of information is obtained from representations of the algebra
onto the upper triangular $2\times2$ matrices.  In that case, the underlying algebra
$A(K)$ has no such representations, and the restriction to $A(K)$ was similar
to a diagonal representation.
Here however, there is a complicating factor because $\fA_n$ has
so many finite dimensional representations.
It will be necessary to use a much larger family of nest representation to
filter out the ``noise'' caused by these other representations.

The free semigroup $\Fn$ consists of all words in an alphabet of $n$ letters.
Consider the left regular representation $\lambda$ of $\Fn$ on Fock space, $\Fock$.
Let $S_i = \lambda(i)$.
The non-commutative disc algebra $\fA_n$, for $n \ge 2$,  is the nonself-adjoint
unital algebra generated by $S_1,\dots,S_n$.
It sits as a subalgebra of the Cuntz--Toeplitz C*-algebra.
However the quotient map onto the Cuntz algebra is completely isometric on $\fA_n$.
So $\fA_n$ may be considered as a subalgebra of $\O_n$.
Moreover the operator algebra generated by any $n$-tuple of isometries
$V_1,\dots,V_n$ with pairwise orthogonal ranges is completely isometrically
isomorphic to $\fA_n$.
These algebras were introduced by Popescu \cite{Pop2} as a natural multivariable
generalization of disc algebra $A(\bD)$.
The Frahzo--Bunce dilation Theorem \cite{Fr,Bun,Pop1} shows that any row contractive $n$-tuple
$T=[T_1, \dots, T_n]$ dilates to an $n$-tuple of isometries $V = [V_1, \dots, V_n]$
with pairwise orthogonal ranges.
Popescu used it to establish a natural analogue of the von Neumann inequality
for row contractive $n$-tuples.
 
The (norm closed) d-shift algebra $\A_d$, $d \ge 2$, plays the same role for a
row contractive $d$-tuple of commuting operators.
It is the quotient of $\fA_d$ by the commutator ideal, and is represented by the
operators $\hat{S_i}$ which act on symmetric Fock space, $\H_d$, obtained as the
compression of $S_i$ to this co-invariant subspace.
Symmetric Fock space can be considered as a space of functions on $\bB_n$.
It is a reproducing kernel Hilbert space on $\bB_n$ with kernel $k(x,y) = \dfrac1{1-\ip{y,x}}$.
The operators $\hat{S_i}$ are just the multipliers by the coordinate functions $z_i$.
This algebra was introduced by Arveson \cite{Arv3}.
The corresponding dilation theorem due to Drury \cite{Dru} shows that any commuting
row contractive $d$-tuple $T$ is the range of a completely contractive homomorphism
of $\A_d$ taking generators to generators.

\section{Crossed Products} \label{S:crossed}

\begin{defn}   \label{def:sem}
Let $\Aut(\fA_n)$ denote the isometric automorphisms
of $\fA_n$, and let $\phi \in \Aut(\fA_n)$. The
semicrossed product $\fA_n \times_\phi \bZ^+$ is the
universal operator algebra containing $\fA_n$ and a contraction $U$
so that $AU = U \phi(A)$ for all $A \in \fA_n$.
\end{defn}

Specifically, for any completely contractive representation $\pi$ of $\fA_n$
on a Hilbert space $\H$ and a contraction $K \in B(\H)$
satisfying $\pi(A)K=K \pi(\phi(A))$, for all $A \in \fA_n$, will produce
a contractive representation $\pi \times K$ of $\fA_n \times_\phi \bZ^+$ on $\H$, which on polynomials
is defined as follows
\[
(\pi \times K)(\sum_n U^n A_n ) = \sum_n K^n \pi(A_n).
\]
There are several such representations
of $\fA_n \times_\phi \bZ^+$, which, in one form or another, have already appeared in the literature.

\begin{eg} Let $\phi \in \Aut \fA_n$.
Voiculescu \cite{V} has constructed a unitary $U$ acting on the
Fock space $\H_n$ which implements the action of $\phi$ on the Cuntz--Toeplitz
$\ca$-algebra $\E_n$ by $U^*AU = \phi(A)$.
This provides a covariant
pair $(\id_{\E_n}, U)$ for $(\E_n, \phi)$ and therefore it
produces a representation $\id_{\E_n}\times U$ of
$\E_n\times_\phi \bZ$.
This provides a representation $(\id_{\fA_n}, U)$ of
$\fA_n \times_\phi \bZ^+$ by restriction.

Similarly, by taking quotients with the compacts on $\H_n$,
we obtain a covariant
representation $(\pi_{\O_n}, \hat{U})$ for $(\O_n, \phi)$ and therefore representations for both
$\O_n \times_\phi \bZ$ and $\fA_n \times_\phi \bZ^+$, inside
the Calkin algebra, which we denote as $\pi_{\O_n}\times \hat{U}$.
In the case where $\phi$ is aperiodic, both representations $\pi_{\O_n}\times \hat{U}$
and $\id_{\E_n}\times U$ of $\fA_n \times_\phi \bZ^+$
are faithful \cite{DavKlast}.
\end{eg}

\begin{eg}
Let $\pi$ be any completely contractive representation of $\fA_n$ on a Hilbert space $\H$.
Define $\tilde\pi$ on $\H \otimes \ltwo$ by
\[ \tilde\pi(a) = \sum_{k\ge0}\strut^\oplus \pi\phi^k(a) \qand U = I \otimes S^* \]
where $S$ is the unilateral shift.  This is easily seen to yield a completely contractive
representation $\tilde\pi$ of $\fA_n$ and a contraction $U$ so that
$(\tilde\pi,U)$ yields a representation of $\fA_n \times_\phi \bZ^+$.
\end{eg}

\begin{eg} Let $\phi \in \Aut \fA_n$.
Consider the non-commutative disc algebra $\fA_{n+1}$ acting on the Fock space
$\H_{n+1}$ and define an ideal
\[ \fJ = \big\lip S_iS_{n+1} -S_{n+1}\phi(S_i) : 1 \le i \le n \big\rip .\]
The \wot-closure $\ol{\fJ}$ of $\fJ$ is an ideal of $\fL_{n+1}$, and these ideals were studied
in \cite{DP2,DP3}.  In particular, it is shown in \cite{DP2} that $\ol{\fJ}$ is determined by
its range, which is a subspace invariant for both $\fL_{n+1}$ and its commutant $\fR_{n+1}$.
Then in \cite{DP3}, it is shown that  $\fL_{n+1}/\ol{\fJ}$ is completely isometrically
isomorphic to the compression to $\H_\phi = \ran(\fJ)^\perp$.  Since
\[
  \ol{\ran{\fJ}} = \spn \big\{  A(S_iS_{n+1} -S_{n+1}\phi(S_i))\H_{n+1} :  i=1, 2, \dots, n \big\}
\]
is evidently orthogonal to $\xi_\mt$, we see that $\H_\phi$ is non-empty.
The compression of $\fA_n$ to $ \H_{\phi}$ is a completely contractive homomorphism $\rho$,
and the compression $B$ of $S_{n+1}$ is also a contraction.  Therefore
$(\rho, B)$ is a covariant representation of $(\fA_n,\phi)$, and thus
determines a completely contractive representation $\fA_n \times_\phi \bZ^+$.
\end{eg}

\begin{eg} \label{used} Any representation of $\fA_n$ produces a
representation of $\fA_n \times_\phi \bZ^+$
by simply taking $U=0$. In section \ref{S:ncdisc}, we will use repeatedly that observation.
In section \ref{sec:main} various finite dimensional
representations of $\fA_n \times_\phi \bZ^+$ will also be constructed that will allow us to
classify them as algebras.
\end{eg}

\begin{eg} In the special case where $\phi$ is just
a permutation, Power \cite{Pow2} and
Davidson, Power and Yang \cite{DPY} have studied a specific
representation of $\fA_n \times_\phi \bZ^+$, whose
image is the tensor algebra of a rank 2 graph. (See also \cite{PSolel}.)
\end{eg}

Actually, this last example raises the question whether one can construct
\textit{concrete} faithful representations of $\fA_n \times_\phi \bZ^+$, for arbitrary $\phi$. The case
$n=1$ has been resolved in \cite{BP} and the general case is being considered in \cite{DavKlast}.

One can also define the semicrossed product $\A_n \times_\phi \bZ^+$ of the d-shift algebra $\A_n$ by an
automorphism $\phi \in \Aut(\A_n)$, as we did in Definition \ref{def:sem}. The above examples,
modified appropriately, will produce representations for that
semicrossed product as well.

\section{The character space of $\fA_n \times_\phi \bZ^+$.}   \label{sec:character}

The character space $\fM_{\fA_n}$ of $\fA_n$ is homeomorphic to $\ol{\bB}_n$
via the identification $\theta \to \theta^{(n)}([S_1\ \dots S_n])$.
Indeed, every multiplicative linear functional is unital and completely contractive.
So $\theta^{(n)}([S_1\ \dots S_n]) \in \ol{\bB}_n$.
The Frahzo--Bunce--Popescu dilation theory \cite{Fr,Bun,Pop1} for row contractions
shows that all of these points are possible for $\fA_n$ and hence for $\A_n$.
Since the polynomials are dense in $\fA_n$ and in $\A_n$, it
follows that $\theta$ is uniquely determined by this data.
So we can write $\theta_z$ for the character corresponding to $z \in \bB_n$.
When $\|z\|<1$, this functional is given as a vector state on the left regular representation.

Any $\phi \in \Aut(\fA_n)$ acts on the character space
$\fM_{\fA_n}$ by $\hat\phi(\theta) = \theta \circ \phi$.
Note that $\hat\phi$ is a homeomorphism, and so carries the interior $\bB_n$ onto itself.
As in Davidson--Pitts \cite{DP2} (the \wot-continuous version), it follow that $\hat\phi$
is a conformal automorphism of $\bB_n$.
See \cite{Pop_auto} for a different proof by Popescu.
It turns out that this assignment is one-to-one and onto and therefore establishes
a bicontinuous isomorphism between $\Aut(\fA_n)$ and $\Aut(\bB_n)$.

\begin{defn}
If $\A$ is a Banach algebra,
a subset $\O$ of $\fM_{\A}$ is said to be \textit{analytic} if there exist
a domain $\Omega \subset \bC^k$ and a homeomorphism $h:\Omega\to\O$
such that $\hat{a}\circ h$ is analytic for every $a \in \A$.
\end{defn}

If $\C$ is the commutator ideal of the noncommutative disk algebra $\fA_n$,
then $\A_n \simeq \fA_n/\C$.  Let $q$ be the quotient map.
If $\hat\phi \in \Aut(\bB_n)$, write $\phi$ for the corresponding automorphism of $\fA_n$
and $\Phi$ for the automorphism of $\A_n$.  Then $q\phi = \Phi q$.
This induces a quotient map of $\fA_n \times_\phi \bZ^+$,
onto $\A_n \times_\Phi \bZ^+$.
As $\C$ is contained in the commutator ideal of $\fA_n \times_\phi \bZ^+$,
we see that $\A_n \times_\Phi \bZ^+$ has  the same character space.
It follows easily that they have the same maximal analytic sets.

For $\phi \in \Aut(\fA_n)$, we write
\[ \Fix(\hat\phi) = \{ \theta \in \fM_{\fA_n} : \hat\phi(\theta) = \theta \}. \]
This consists either of an affine subset of $\ol{\bB}_n$, or one or two points
on the boundary (see Rudin \cite{Rud}).  Let $F_0(\phi) = \Fix(\phi) \cap \bB_n$ and
let $F_1(\phi)= \Fix(\phi) \cap \partial\bB_n$.
The next result characterizes the maximal analytic subsets of $\fM_{\fA_n \times_\phi \bZ^+}$.

\begin{lem}  \label{useful}
$\fM_{\fA_n \times_\phi \bZ^+} =( \ol{\bB}_n \times \{0\}) \cup (\Fix(\phi) \times \ol{\bD})$.
The maximal analytic sets in $\fM_{\fA_n \times_\phi \bZ^+}$ are
$\bB_n\times\{0\}$, $F_0(\phi) \times \bD$,
$F_0(\phi) \times \{\lambda\}$ for $\lambda \in \bT$,
and $\{x\} \times \bD$ for $x \in F_1(\phi)$.
\end{lem}
 
\begin{proof}
A character $\theta$ is determined by its restriction to $\fA_n$, $\theta|_{\fA_n}$,
and the scalar $\lambda = \theta(U)$.
There is a point $z \in \ol{\bB}_n$ so that $\theta|_{\fA_n} = \theta_z$,
and $|\lambda| \le \|U\|=1$.
As in Davidson--Katsoulis \cite[\secsymb3]{DKconj}, if $\theta_z$ is not a fixed point of $\hat\phi$, then $\lambda = 0$.
In that case, we write $\theta = \theta_{z,0}$.
On the other hand, if $\theta_z$ is a fixed point of $\hat\phi$, then the 1-dimensional
representation $\theta_z$ of $\fA_n$ and any $\lambda \in \ol{\bD}$ satisfy the covariance
relations.
Therefore there is a (completely contractive) representation $\theta_{z,\lambda}$
of the semicrossed product
$\fA_n \times_\phi \bZ^+$ on $\bC$ with $\theta_{z,\lambda}|_{\fA_n} = \theta_z$
and $\theta_{z,\lambda}(U)=\lambda$.
Clearly $\theta_{z,\lambda}$ is a character.  Thus
\[
 \fM_{\fA_n \times_\theta \bZ^+} =
 ( \ol{\bB}_n \times \{0\}) \cup (\Fix(\phi) \times \ol{\bD}) .
\]

Evidently the sets $\bB_n\times\{0\}$, $F_0(\phi) \times \bD$,
$F_0(\phi) \times \{\lambda\}$ for $\lambda \in \bT$, and
$\{x\} \times \bD$ for $x \in F_1(\phi)$ are analytic sets.
On the other hand, suppose that $\Omega \subset \bC^m$
is a domain and  $h:\Omega \to \fM_{\fA_n \times_\phi \bZ^+}$
is an analytic map.
Either $h(\Omega)$ contains $\theta_{x,0}$ for some $x \not\in \Fix(\phi)$,
or $h(\Omega) \subset \Fix(\phi) \times \ol{\bD}$.

In the first case, suppose that $h(w_0) = \theta_{x,0}$.
Then there is a neighbourhood $\W$ of $w_0$
so that $h(w) = \theta_{h_1(w),0}$ for $w \in \W$.  Hence $h(w)(U) = 0$ on $\W$.
By analyticity, it must vanish on all of $\Omega$.
So $h(\Omega) \subset \ol{\bB}_n \times\{0\}$.
This function is not constant, so it cannot achieve its maximum modulus.
So $h(\Omega) \subset \bB_n \times \{0\}$.

In the latter case, suppose that it contains point
$\theta_{x,\lambda}$ for $x \in F_0(\phi)$.  Then
\[ g(w) = h(w)\big([S_1\ \dots \ S_n] \big) \qfor w \in \Omega \]
takes values in $\ol{\bB}_n$ and takes a value in $\bB_n$ at some point.
This map is analytic, and hence is either constant or fails to
attain its maximal modulus.  In either case, it remains in $\bB_n \times \ol{\bD}$.
Similarly, either $h(w)(U)$ takes values in $\bD$ or it is a constant
value of modulus $1$.
So it is contained in $F_0(\phi) \times \bD$ or in $F_0(\phi) \times \{\lambda\}$ for some $\lambda \in \bT$.

Finally, suppose that it contains a point $\theta_{x,\lambda}$ for $x \in F_1(\phi)$,
say $x = (x_1,\dots,x_n) \in \partial\bB_n$.
Then $S = \sum_{i=1}^n \ol{x_i} S_i$ is a contraction such that $\theta_{x,\lambda}(S)=1$.
Since $|\theta(S)| \le 1$ for all characters $\theta$ and $h(w)(S)$ is analytic,
it follows that this map is constant.  So $h(\Omega)$ is contained in $\{x\} \times \ol{\bD}$.
Again the maximum modulus theorem shows that $h(\Omega)$ is contained in $\{x\}\times\bD$.
\end{proof}

The structure of the maximal analytic sets for
$\fA_n \times_\phi \bZ^+$ and for $\A_n \times_\Phi \bZ^+$, $n \ge 2$
is much richer than that of the case $n=1$.
In order to handle this situation, we will use the theory of several
complex variables.

\section{The ideal $\fJ_{\phi}$}\label{S:ncdisc}

As in \cite{DKconj,DKmult} where we studied various semicrossed products,
we will make significant use of nest  representations.
These are representations into $\fT_k$, the upper triangular $k \times k$ matrices,
such that the only invariant subspaces of the image are those of $\fT_k$.
However, in these earlier cases, where the algebra was a semicrossed product of $\rC(X)$ or
of $\A(\bD)$, the image of this subalgebra was abelian and in the case of $\rC(X)$,
the  image was diagonalizable.
There is a considerable difference here because there is a plethora of
nest representations of $\fA_n$ onto $\fT_k$ for any $k\ge1$ \cite{DP2,DKnest}.
Consequently, some care is required to separate these from the representations
we require to determine the map $\phi$.
This section provides the technical tool to accomplish that.

\begin{defn}
If $\phi$ is an automorphism of the non-commutative disc algebra $\fA_n$, $n \geq 2$,
then $\fJ_\phi = U(\fA_n \times_\phi \bZ^+)$ will denote the closed ideal of
$\fA_n \times_\phi \bZ^+$ generated by $U$.
\end{defn}

The main result of this section shows that $\fJ_{\phi}$ is determined intrinsically
within $\fA_n \times_\phi \bZ^+$.
To accomplish this, we require a construction of nest representations which specify the
diagonal in advance.
We do this by modifying an earlier construction of ours \cite{DKnest} regarding nest
representations of $\fA_n$.
Then we compose these representations with the canonical quotient
$\pi: \fA_n \times_\phi \bZ^+ / \fJ_{\phi} \simeq \fA_n$, which sends $U$ to $0$,
and obtain nest representations of $\fA_n \times_\phi \bZ^+$.
 
Let $\rep_{\fT_k} \fA_n \times_\phi \bZ^+$ denote the collection of all
representations of $\A$ onto $\fT_k$, the upper triangular $k \times k$ matrices.
To each nest representation $\pi \in \rep_{\fT_k} \fA_n \times_\phi \bZ^+$,
we associate $k$ characters $\theta_{\pi, 1}, \theta_{\pi, 2}, \dots, \theta_{\pi, k}$
which correspond to the diagonal entries.
That is, if $\{\xi_1,\xi_2,\dots,  \xi_k\}$ is the canonical basis of $\bC^k$
which makes $\fT_k$ upper triangular, then
\[
  \theta_{\pi, i}(A)\equiv \langle \pi(A)\xi_i ,  \xi_i\rangle,
  \qfor A \in \fA_n \times_\phi \bZ^+ ,\ i=1,2,\dots, k .
\]

If $Z = (z_1,z_2,\dots ) \in \bB_{n}^{\infty}$, then
$\N_{Z,k}\subseteq \rep_{\fT_k} \fA_n \times_\phi \bZ^+$ will denote the
family of $k\times k$ nest representations $\pi$ such that
$\theta_{\pi,i}=\theta_{z_i,0}$ for $1 \le i \le k$.
We set $\N_Z = \bigcup_{k\ge1} \N_{Z,k}$.

\begin{lem} \label{L:zeroU}
Let $Z\in \bB_{n}^{\infty}$ such that $Z \cap \hat\phi(Z) = \mt$.
If $\rho$ is a $k\times k$ nest representation of $\fA_n \times_\phi \bZ^+$
in $\N_{Z, k}$, then $\rho(U) = 0$.
\end{lem}

\begin{proof}
If $\rho(U) \ne 0$, then there is a non-zero matrix entry $\rho(U)_{ij}$ with $j-i$ minimal.
Since $\hat\phi(z_j) \ne z_i$, there is an element $A \in \fA_n$ such that
$\hat{A}(z_i)=1$ and $\widehat{\phi(A)}(z_j) = \hat{A}(\hat\phi(z_j)) =0$.
Then using the fact that $\rho(U)_{ik}=0=\rho(U)_{kj}$ for $i \le k < j$, compute
\begin{align*}
 0 &=
 \rho(AU-U\phi(A))_{ij} = \theta_{z_i,0}(A) \rho(U)_{ij} - \rho(U)_{ij} \theta_{z_j,0}(\phi(A))
 \\&= \hat{A}(z_i)  \rho(U)_{ij} - \rho(U)_{ij}  \hat{A}(\hat\phi(z_j))  =  \rho(U)_{ij}. \qedhere
\end{align*}
\end{proof}

Since $\hat\phi$ is an analytic function, the property that $Z \cap \hat\phi(Z) = \mt$
is generic for sequences $Z\in \bB_{n}^{\infty}$.
We would like to show that in this case, that $\J_Z := \bigcap_{\rho\in\N_Z} \ker\rho$ is
equal to $\fJ_{\phi}$.  While this may be true, we were not able to prove it.
However, we show that if one allows the diagonal entries to vary over small neighbourhoods,
then this is the case (generically).

Start with a sequence $Z= (z_1, z_2, \dots ) \in \bB_{n}^{\infty}$ and
a word $w=l_1\dots l_k \in \Fn$ of length $k\ge2$.
Fix $0 \le \delta <1$ so that $Z_{k+1}\equiv (z_1, z_2, \dots ,z_{k+1}) \in \delta \bB_{n}^{k+1}$ and
construct a representation $\rho_{Z ,w}$ in $\N_{Z,k+1}$ as follows.
Write $z_i = (z_{i1},\dots,z_{in})$.
Define $\rho_{Z,w}(U)=0$ and
\[ \rho_{Z,w}(S_j) = \diag(z_{ij})_{i=1}^{k+1} + (1-\delta) \sum_{l_i=j} E_{i,i+1} \qfor 1 \le j \le n .\]
It is easy to see that
\[ \big\| \big[ \rho_{Z,w}(S_1) \ \dots \ \rho_{Z,w}(S_n) \big] \big\| < 1 .\]
So this determines a completely contractive representation.

Evidently, if we fix the word $w$, then the above construction not only works
for $Z$ but also for any other (finite or infinite)
sequence $Z'$, as long as
$Z_{k+1}^{\prime} \in \delta \bB_{n}$. Therefore, whenever convenient,
we will think of $Z$ in $\rho_{Z,w}$ as a variable that ranges over all sequences
whose first $k+1$ entries belong to $\delta \bB_{n}$.

\begin{lem}
With the above notation, we have $\rho_{Z,w}(S_w)_{1,k+1} = (1-\delta)^k$
and $\rho_{Z,w}(S_v)_{1,k+1} = 0$ for
all other words with $|v| \le k$.  For each word $v$ with $|v|>k$, there is an analytic
function $F_v$ on $\delta \bB_n^{k+1}$ with $F(0)=0$
so that
\[ \rho_{Z,w}(S_v)_{1,k+1} = F_v(z_1,\dots,z_{k+1}) .\]

Moreover, if the coordinates $\{ z_{ij} : 1\leq i \leq k+1, 1 \le j \le n \}$ are all distinct,
then the range of $\rho_{Z,w}$ is all of $\T_{k+1}$. So $\rho_{Z,w} \in \N_{Z,k+1}$.
\end{lem}

\begin{proof}
Observe that $\rho_{Z,w}(S_v)_{1,k+1} = 0$ unless $v$ has the form
\[ v = v_1 l_1 v_2 l_2 \dots l_k v_{k+1} \qfor v_i \in \Fn .\]
Of course, for words of length greater than $k+1$, there may be more than one
such factorization of $v$.
Considering $Z=(z_1, z_2, \dots, z_{k+1})$ with $z_i$ contained in $\delta\bB_{n}^{k+1}$, we
see that the matrix coefficient $ \rho_{Z,w}(S_v)_{1,k+1}$ is the sum of terms of the form
\[
  (1-\delta)^k \prod_{1 \le i \le k+1} \theta_{z_i}(S_{v_i})
\]
Each term is a  monomial of degree $\sum |v_i|  = |v|-k-1$.
In particular, $\rho_{Z,w}(S_w)_{1,k+1} = (1-\delta)^k$.
For $|v|>k+1$, this is a homogeneous polynomial $F_v(z_1,\dots,z_{k+1})$;
so $F_v(0)=0$.

Assume now that the coordinates $\{ z_{ij} :  1\leq i \leq k+1, 1 \le j \le n \}$ are all distinct.
First we show that $E_{1,1}$ is in the range.  Take any $j \ne l_1$.
Pick a polynomial $p$ so that $p(z_{1j}) = 1$ and $p(z_{ij}) = 0$ for $2 \le i \le k+1$.
This is possible since $z_{ij}$ are distinct.
Then $\rho_{Z,w}(S_j) = \diag(z_{ij})_{i=1}^{k+1} + X$ where $X$ is supported on the
first superdiagonal, and since $j\ne l_1$, $X_{1,2}=0$; so $X=E_{1,1}^\perp X E_{1,1}^\perp$.
Therefore, $\rho_{Z,w}(p(S_j)) = E_{1,1} + T$ where $T = E_{1,1}^\perp T E_{1,1}^\perp$
is strictly upper triangular.  Hence $\rho_{Z,w}(p(S_j)^k) = E_{11}$.

One can now similarly show that $E_{i,i}$ lies in the range of $\rho_{Z,w}$ by induction.
It follows easily that each $E_{i,i+1}$ lies in the range of $\rho_{Z,w}$.
Hence the range is all of $\T_{k+1}$.
\end{proof}

The next step is to show that for elements $A\ne 0$ in $\fA_n$,
there are enough representations which are non-zero on $A$.

\begin{lem} \label{L:Znest_reps}
Let $Z \in \bB_{n}^{\infty}$. For any $0 \ne A \in \fA_n$ and any $\ep_i>0$, there
is a word $w$ of length $k\ge2$ and
a sequence $Z' \in \bB_{n}^{k+1}$
with coordinates $\{ z'_{ij} : 1 \leq i \leq k+1, 1 \le j \le n \}$ all distinct
satisfying $\|z'_i-z_i\| < \ep_i$ so that $\rho_{Z',w}(A) \ne 0$.
\end{lem}

\begin{proof}
If $A$ is not in the commutator ideal, then $\hat A$ is a non-zero analytic function on $\bB_n$.
Thus it has only isolated zeros.  Therefore one can modify $z_1$ by a sufficiently small
perturbation to satisfy all of the requirements and make $\rho_{Z',w}(A)_{1,1} \ne 0$ for
arbitrary choice of $w$.

So suppose that $A \in \C$.
Then the Fourier series $\sum a_w S_w$ of $A$ has no constant or linear terms.
Let $w$ be a word of minimal length such that $a_w \ne 0$.
Then by the previous lemma, $\rho_{Z,w}(S_w) = (1-\delta)^k$, for an appropriate $0\le\delta\le1$, and
$[\rho_{Z,w}(A)]_{1,k+1} = a_w (1-\delta)^k + F(z_1,\dots,z_{k+1})$,
where $F$ is an analytic function, defined on $\delta\bB_{n}^{k+1}$, with 0 constant term.
This matrix entry might vanish at $(z_1,\dots,z_{k+1})$, but will then necessarily
be non-constant.
So it will be non-zero at most points $(z'_1,\dots,z'_{k+1})$ arbitrarily
near to $(z_1,\dots,z_{k+1})$.  So one may choose $Z'$ accordingly.
\end{proof}

If $Z=(z_1, z_2, \dots)$ and $Z'=(z'_1, z'_2, \dots, z'_l)$ belong
to $\bB_{n}^{\infty}$ and $\bB_{n}^{l}$ respectively, then
we write $d(Z,Z') := \sup\{ \|z_i-z'_i\| : 1\le i \le  l\}$.

We come to the main result of this section, which shows that $\fJ_\phi$
is intrinsically defined.

\begin{thm}  \label{nice}
If $Z\in \bB_{n}^{\infty}$, then the ideal
\[
 \fI_Z := \bigcap_{l \ge2} \bigcup_{\ep>0}
 \bigcap \{ \ker \rho : \rho\in\N_{Z',l},\ Z'\in\bB_{n}^{l},\ d(Z',Z)<\ep\ \}
\]
equals $\fJ_{\phi}$ whenever $\hat\phi(Z) \cap Z = \mt$;
and it is always contained in $\fJ_{\phi}$. Therefore
$\fJ_{\phi}=\bigcup\fJ_Z$.
\end{thm}

\begin{proof}
If $\hat\phi(Z) \cap Z = \mt$, then for each fixed $l$ and sufficiently small $\ep>0$,
one has $\hat\phi(\{z'_1,\dots,z'_l\}) \cap (\{z'_1,\dots,z'_l\}) = \mt$ when $d(Z',Z)<\ep$.
For all $\rho\in\N_{Z',l}$, one has $\rho(U)=0$ by Lemma~\ref{L:zeroU}.
So $\ker\rho$ contains $\fJ_{\phi}$. Hence $\fI_Z$ contains $\fJ_{\phi}$.

Conversely, if $X \not\in \fJ_{\phi}$, then $q(X) = A \ne 0$.
Lemma~\ref{L:Znest_reps} shows that there is a $k \geq 2$ and a word $w$
of length $k$ so that for any $\ep>0$, there is a $Z'$ in $\bB_{n}^{k+1}$ with $d(Z',Z)<\ep$ and
$\rho_{Z',w} \in \N_{Z',k+1}$ so that $\rho_{Z',w}(X) \ne 0$.
Hence for $l=k+1$, we have
\[
X \notin \bigcup_{\ep>0}
 \bigcap \{ \ker \rho : \rho\in\N_{Z',l},\ Z'\in\bB_{n}^{l},\ d(Z',Z)<\ep\ \}
\]
and therefore $X \notin \fI_Z$.
So $\fI_Z = \fJ_{\phi}$.

Since $\fI_Z = \fJ_{\phi}$ we obtain $\fJ_{\phi}  \subset \bigcup \fI_Z$.
Even if $\hat\phi(Z) \cap Z \ne \mt$, for any $l\ge1$, there will always be $Z'\in\Z$
with $d(Z',Z)<\ep$ and $\hat\phi(\{z'_1,\dots,z'_l\}) \cap (\{z'_1,\dots,z'_l\} = \mt$.
So using Lemma~\ref{L:Znest_reps} once again, we see that if $X \not\in\fJ_{\phi}$, then $X \not\in\fI_Z$.
So $\fI_Z$ is always a subset of $\fJ_{\phi}$.
\end{proof}

\section{The main result.}   \label{sec:main}

If $\gamma:\fB_1 \rightarrow \fB_2$
is an isomorphism between algebras $\fB_1$ and $\fB_2$, then $\gamma$ induces isomorphisms,
\begin{alignat*}{2}
\gamma_c&:\M_{\fB_1} \to \M_{\fB_2}
&\quad\text{by}\quad& \gamma_c(\theta) =  \theta\circ\gamma^{-1}
\\ \gamma_r&:\rep_{\fT_2}  \fB_1 \to \rep_{\fT_2}  \fB_2
&\quad\text{by}\quad& \gamma_r(\pi) =  \pi\circ\gamma^{-1}
\end{alignat*}
which are compatible in the sense that
\[
 \gamma_c (\theta_{\pi, i}) = \theta_{\gamma_r(\pi),i}
 \qfor i=1,2 \AND \pi \in \rep_{\fT_2}  \fB_1.
\]

 Let $\phi_1$ and $ \phi_2$ be isometric automorphisms of $\fA_n$ and assume that
 $\gamma : \fA_n \times_{\phi_1} \bZ^+   \rightarrow \fA_n \times_{\phi_2} \bZ^+$
is an isomorphism. We can now define a third map $\gamma_s :  \bB_n\rightarrow \bB_n$
 associated with $\gamma$, by the formula
 \[
 \gamma_s(z) = (\theta_{z,0}\circ \gamma^{-1})^{(n)}[S_1, S_2, \dots S_n], \qquad z \in \bB_n.
 \]
An easy approximation argument shows that $\gamma_s$ is a holomorphic map.
 
\begin{lem} \label{Ligo}
Let $\phi_1 , \phi_2$ be non-trivial isometric automorphisms of $\fA_n$.
Assume that there exists an isomorphism
$\gamma : \fA_n \times_{\phi_1} \bZ^+   \rightarrow \fA_n \times_{\phi_2} \bZ^+$.
If $\gamma_{s}$ and $\gamma_c$ are defined as above, then
\[
 \gamma_c(\bB_n\times \{0\} )=  \bB_n\times \{0\}
\]
Hence, $\gamma_s^{-1}=(\gamma^{-1})_s$ and so
$\gamma_s$ is a biholomorphism.
Furthermore, the fixed point sets in the open ball, $F_0(\hat\phi_i)$, satisfy
\[
 \gamma_s((F_0(\hat\phi_1)) = F_0(\hat\phi_2).
\]
\end{lem}

\begin{proof} As $\phi_i \ne \id$, Lemma~\ref{useful} shows that $\bB_n \times \{0\}$
is distinguished from the other maximal analytic subsets of the maximal ideal
space of $\fA_n \times_{\phi_i} \bZ^+$ as
the only maximal analytic set of dimension $n$ which is not a product domain.
A result by Ligocka \cite[Theorem 2]{Lig} states that a biholomorphic map
between products of domains with boundaries of class $C^2$ (which our sets
clearly have) must preserve the number of factors. Hence
$\gamma_c(\bB_n \times \{0\} )= \bB_n \times \{0\}$.  It follows that
$\gamma_s^{-1}=(\gamma^{-1})_s$, i.e., $\gamma_s$ is a biholomorphism.

Now $F_0(\hat\phi_i)$ is determined as the intersection of $\bB_n \times \{0\}$
with the union of all other maximal analytic sets.
It follows that
\[ \gamma_s((F_0(\hat\phi_1)) = F_0(\hat\phi_2) .\qedhere \]
\end{proof}

We now show that $\fJ_{\phi}$ is invariant under isomorphisms.

\begin{cor}   \label{usefultoo}
Assume that $\phi \ne \id$ is an automorphism of $\bB_n$.
Then $\fJ_{\phi} = U (\fA_n \times_\phi \bZ^+)$ is an algebraic invariant of
$\fA_n \times_\phi \bZ^+$ given by $\fJ_{\phi} = \bigcup \fI_Z$.
\end{cor}

\begin{proof}
The invariance of $\bB_n\times \{0\}$ shows that sequences
in $\bB_{n}^{\infty}$ are intrinsically defined in the character space of
$\fA_n \times_\phi \bZ^+$. The proof follows now from Theorem \ref{nice}
and the uniform continuity of $\gamma_c$.
\end{proof}

The above results allow us to focus on a special class of nest representations.
Let $\phi$ be an automorphism of the non-commutative disc algebra $\fA_n$.
We denote by $\srep_{\fT_2} \fA_n \times_\phi \bZ^+$ the subset of
$\rep_{\fT_2} \fA_n \times_\phi \bZ^+$ consisting of all nest representations
$ \pi \in \rep_{\fT_2} \fA_n \times_\phi \bZ^+$
such that $\theta_{\pi, 1}, \theta_{\pi, 2} \in \bB_n\times \{ 0\}$,
$\theta_{\pi, 2}|_{\fA_n} \notin \Fix(\hat\phi)$ and $\pi(U)\neq 0$.
By Lemma \ref{useful} and Corollary \ref{usefultoo},
$\srep_{\fT_2} \fA_n \times_\phi \bZ^+$ is preserved by the
dual action of isomorphisms between semicrossed products of the
non-commutative disc algebra $\fA_n$.

Let $\pi \in \srep_{\fT_2} \fA_n \times_\phi \bZ^+$ with
$\pi(U) =  \left(\begin{smallmatrix} 0&c\\0&0 \end{smallmatrix}\right)$, $c\neq0$.
Applying $\pi$ now to the covariance relation
\[
S_iU =U\phi(S_i), \quad 1\leq i \leq n,
\]
and comparing $(1,2)$ entries, we conclude that $\theta_{\pi, 1} =
\hat{\phi}(\theta_{\pi, 2} )$. The situation is parallel to the
$n=1$ case and the techniques of \cite{DKconj} are now applicable.

\begin{thm} \label{main}
Let $\phi_1$ and $\phi_2$ be automorphisms of the
non-comm\-uta\-tive disc algebra $\fA_n$, $n \geq 2$.  Then
the semicrossed products $\fA_n \times_{\phi_1} \bZ^+$ and $\fA_n \times_{\phi_2} \bZ^+$ are isomorphic as algebras
if and only if $\phi_1$ and $\phi_2$ are conjugate
via an automorphism of $\fA_n$.
\end{thm}

\begin{proof}
Firstly, one can recognize $\phi = \id$ from the fact that $\bB_n \times \bD$ is the unique
maximal analytic set of largest dimension.  So we may suppose that $\phi_i \ne \id$.

Assume that there exists an isometric isomorphism
\[
\gamma: \fA_n \times_{\phi_1} \bZ^{+}\longrightarrow \fA_n \times_{\phi_2} \bZ^{+}.
\]
Then by Corollary~\ref{usefultoo}, $\gamma$ induces a map $\tilde\gamma$ of
$\fA_n \simeq \fA_n \times_{\phi_1} \bZ^+/ \fJ_{\phi_1}$ onto
$\fA_n \simeq \fA_n \times_{\phi_2} \bZ^+/ \fJ_{\phi_2}$.
This in turn determines a conformal automorphism $\hat\gamma$ of $\bB_n$.
We will show that $\hat\gamma \circ\phi_1=  \phi_2 \circ \hat\gamma$ by establishing that
\begin{equation} \label{important}
\theta \circ \gamma^{-1} \circ\phi_2=\theta \circ \gamma^{-1}  =
\theta \circ \phi_1\circ \gamma^{-1}
\qforal \theta \in \bB_n .
\end{equation}
If $\theta \in F_0(\phi_1)$, then by Lemma \ref{useful},
$\theta \circ \gamma^{-1} \in F_0(\phi_2)$.
Hence,
\[
\theta \circ \gamma^{-1} \circ\phi_2=\theta \circ \gamma^{-1}  =
 \theta \circ \phi_1\circ \gamma^{-1},
\]
and so (\ref{important}) is valid in that case.

Assume that $\theta \notin F_0(\phi_1)$. A standard dilation argument (and
the universality of the semicrossed product) imply
the existence of a representation
$\pi \in \srep_{\fT_2} \fA_n \times_{\phi_1} \bZ^+$ so that
$\theta_{\pi, 2} = \theta$. But then
\[
\pi\circ \gamma^{-1} \in  \srep_{\fT_2} \fA_n \times_{\phi_1} \bZ^+.
\]
Hence, by our earlier discussion
\[
\theta_{\pi\circ\gamma^{-1}, 2}\circ\phi_2= \theta_{\pi \circ \gamma_{-1}, 1}.
\]
Since
\[
\theta_{\pi\circ\gamma^{-1}, 2}= \theta_{\pi, 2}\circ \gamma^{-1}
\]
and
\[
\theta_{ \pi\circ \gamma^{-1}, 1} =  \theta_{ \pi , 1}\circ\gamma^{-1}
= \theta_{ \pi , 2}\circ \phi_1 \circ \gamma^{-1}
\]
we once again obtain (\ref{important}) and the proof is complete.
\end{proof}

\section{Semicrossed products of $\A_d$}

A result similar to that of Theorem \ref{main} also holds in the case
of the $d$-shift algebras.

\begin{thm}  \label{secondmain}
If $\phi_1$ and $\phi_2$ are automorphisms of the
norm closed $d$-shift algebra $\A_d$, $d \geq 2$, then
the semicrossed products $\A_d \times_{\phi_1} \bZ^+$ and
$\A_d \times_{\phi_2} \bZ^+$ are isomorphic as algebras
if and only if $\phi_1$ and $\phi_2$ are conjugate
via an automorphism of $\A_d$.
\end{thm}

In order to prove Theorem \ref{main} we showed that the ideal $\fJ_{\phi}$ is an
isomorphism invariant. We then defined as  $\srep_{\fT_2} \fA_d \times_\phi \bZ^+$ to be
the class of $2 \times 2$ nest representations of $\fA_d \times_{\phi_1} \bZ^+$
which do not annihilate $\fJ_{\phi}$. Theorem \ref{main}
implies then that $\srep_{\fT_2} \fA_d \times_\phi \bZ^+$
is invariant under the dual actions of isomorphisms.

The proof of Theorem \ref{secondmain} requires much less. Since $\A_d$ is commutative,
any $2 \times 2$ nest representation of $\A_d \times_{\phi_1} \bZ^+$ could not
annihilate the universal contraction $U$, for then it will not be surjective.
Hence repeating mutatis mutandis the arguments in the proof of
Theorem \ref{main}, we get a proof for Theorem \ref{secondmain} as well.
(Recall that the results of Section \ref{sec:character}
have been established for both $\fA_n \times_{\phi_1} \bZ^+$ and
$\A_d \times_{\phi_1} \bZ^+$.)


\end{document}